%% file: psdsampta.tex
\documentclass[conference]{IEEEtran}

\input{\jobname-header.tex}

\usepackage{cite}

\usepackage{url}

\usepackage{amssymb,amsthm}

\usepackage{mathtools}

\usepackage{tikz}

\newcommand{\Real}{\mathbb{R}}
\newcommand{\Integer}{\mathbb{Z}}

\newcommand{\Xf}{\mathbf{X}}
\newcommand{\Yf}{\mathbf{Y}}
\newcommand{\Zf}{\mathbf{Z}}

\newcommand{\XfN}{\Xf^{(N)}}
\newcommand{\YfN}{\Yf^{(N)}}

\newcommand{\Phat}{\widehat{P}}

\newcommand{\vzero}{\vec 0}

\newcommand{\vxi}{\vec \xi}
\newcommand{\vi}{\vec \imath}
\newcommand{\vj}{\vec \jmath}
\newcommand{\vk}{\vec k}
\newcommand{\vl}{\vec l}
\newcommand{\va}{\vec a}

\def\expect{\mathbb{E}}

\def\cov{\mathrm{Cov}}

\def\euler{\mathrm{e}}
\def\ramuno{\mathrm{i}}

\newcommand\rect{\mathrm{rect}}

\newcommand\MT{\mathrm{(MT)}}
\newcommand\MTlinear{\mathrm{(MT,lin)}}

\newtheorem{definition}{Definition}
\newtheorem{lemma}{Lemma}
\newtheorem{theorem}{Theorem}

\begin{document}

\title{Factor Analysis for Spectral Estimation}

\author{
\IEEEauthorblockN{Joakim And\'en}
\IEEEauthorblockA{Program in Applied and Computational Mathematics\\
Princeton University\\
Princeton, NJ, 08544\\
Email: janden@math.princeton.edu}
\and
\IEEEauthorblockN{Amit Singer}
\IEEEauthorblockA{Department of Mathematics and\\Program in Applied and Computational Mathematics\\
Princeton University\\
Princeton, NJ, 08544}
}

\maketitle

\begin{abstract}
Power spectrum estimation is an important tool in many applications, such as the whitening of noise.
The popular multitaper method enjoys significant success, but fails for short signals with few samples.
We propose a statistical model where a signal is given by a random linear combination of fixed, yet unknown, stochastic sources.
Given multiple such signals, we estimate the subspace spanned by the power spectra of these fixed sources.
Projecting individual power spectrum estimates onto this subspace increases estimation accuracy.
We provide accuracy guarantees for this method and demonstrate it on simulated and experimental data from cryo-electron microscopy.%
\footnote{The authors were partially supported by Award Number R01GM090200 from the NIGMS, FA9550-12-1-0317 from AFOSR,
Simons Investigator Award and Simons Collaboration on Algorithms and Geometry from Simons Foundation, and the Moore Foundation Data-Driven Discovery Investigator Award.}
\end{abstract}

\IEEEpeerreviewmaketitle

\section{Introduction}
\label{sec:intro}

Estimating the power spectrum of a stationary field arises in many applications \cite{stoica1997introduction,bond1997pervasive,harig2012mapping,delorme2004eeglab}.
For example, noise statistics play an important role in maximum likelihood estimation of signal parameters.
Hence, its power spectrum must be estimated.

The most basic estimator of the power spectrum is the periodogram, which is asymptotically unbiased but has very high variance.
To remedy this, spectral smoothness constraints or parametric models are often imposed \cite{brockwell-davis,percival-walden}.
The multitaper estimator \cite{thomson,percival-walden,abreu2017mse} is particularly popular, which multiplies the data with fixed tapers, computes their periodograms, and averages.
This reduces variance at the expense of smoothing the estimated power spectrum, potentially increasing bias.

In this work, we consider the problem of estimating the power spectra of multiple independent signals.
Given identically distributed signals, an ensemble average of their periodograms or multitaper estimates provides a low-variance power spectrum estimate.
However, in certain applications, signals are not identically distributed, but combine a small number of identically distributed stochastic sources.
For example, a non-stationary signal may be divided up into approximately stationary parts, each of which is a linear combination of some fixed random signals.
Alternatively, noise signals may arise from measurements subject to differing experimental conditions which are described by combining various noise sources.
This is the case for cryo-electron microscopy (cryo-EM) images, where large molecules are frozen in a thin layer of ice and then imaged by exposing them to an electron beam and recording the transmitted electrons \cite{frank,callaway}.
Variations in experimental conditions result in different noise characteristics.
To whiten the projection images, accurate estimates of their noise power spectrum are therefore necessary.

While traditional methods such as multitaper estimates are applied in this context, an estimator which takes into account the low-dimensional variability of the noise distributions could yield improved accuracy.
We propose a factor analysis model in which the $d$-dimensional field $\Xf: \Integer^d \rightarrow \Real$ is a linear combination of $r$ fixed fields
\begin{equation}
\label{eq:model-def}
\Xf[\vi] \coloneqq \sum_{l=1}^r a_l \Zf_l[\vi] \quad \vi \in \Integer^d\mbox{,}
\end{equation}
where $a_1, \ldots, a_r$ are independent random variables and $\Zf_1, \ldots, \Zf_r$ are independent linear random fields.
Typical values for $d$ include $d = 1$ for stationary processes and $d = 2$ for random images.

It follows that the power spectrum of $\Xf$ conditioned on $a_1, \ldots, a_r$ is a linear combination of the power spectra of $\Zf_1, \ldots, \Zf_r$.
That is, the conditional power spectra reside in a low-dimensional subspace and are therefore determined by a small number of factors.
We introduce a method for estimating this subspace and these factors given multiple realizations of $\Xf$.
The method allows for estimation of the number of factors, which is small for most applications.
Linear projection of individual power spectrum estimates onto the factor subspace then yields improved accuracy.

Section \ref{sec:single} discusses the power spectrum estimation problem for a single signal and describes the basic periodogram and multitaper methods.
The factor analysis model and associated power spectrum estimation method are introduced and analyzed in Section \ref{sec:factor}.
Finally, Section \ref{sec:results} presents numerical results on simulations and experimental data taken from cryo-EM images.
Software implementing our proposed method and reproducing the figures of this paper is available at \url{http://github.com/janden/fase/}.

\section{Single Power Spectrum Estimation}
\label{sec:single}

Given a zero-mean stationary random field $\Yf: \Integer^d \rightarrow \Real$ with finite second moments, we define its autocovariance as
\begin{equation}
\label{eq:autocov-def}
R_{\Yf}[\vj] \coloneqq \expect\left[\Yf[\vi]\Yf[\vi+\vj]\right] \quad \vj \in \Integer^d.
\end{equation}
Its Fourier transform is the power spectrum of $\Yf$:
\begin{equation}
\label{eq:powspec-def}
P_{\Yf}(\vxi) \coloneqq \sum_{\vi\in\Integer^d} R_\Yf[\vi] \euler^{-2\pi \ramuno\langle\vxi,\vi\rangle} \quad \vxi \in [-1/2, 1/2]^d.
\end{equation}
The power spectrum $P_{\Yf}$ is symmetric and non-negative.

While $\Yf$ is defined over $\Integer^d$, we are only given samples over some finite domain $M_N^d$, where $M_N = \{-\lceil N/2 \rceil+1, \ldots, \lfloor N/2 \rfloor\}$.
Let us denote this restriction of $\Yf$ to $M_N^d$ by $\YfN$.
Our task is then to estimate $P_{\Yf}$ given $\YfN$.

\subsection{Periodogram}
\label{sec:periodogram}

The most basic estimator for the power spectrum of a stationary field is the periodogram, given by
\begin{equation}
\label{eq:phat-def}
\Phat_{\YfN}[\vk] = \frac{1}{N^d} \left|\sum_{\vi\in M_N^d} \YfN[\vi] \euler^{-2\pi \ramuno \langle \vk, \vi \rangle/N}\right|^2 \quad \vk \in M_N^d\mbox{.}
\end{equation}
Since $\YfN$ is real, $\Phat_{\YfN}$ is symmetric and only needs to be computed on half of $M_N^d$.
Let us therefore choose a subdomain $M_{N,+}^d$ of $M_N^d$ such that $M_{N,+}^d \cup -M_{N,+}^d = M_N~(\mathrm{mod}~N)$ and $M_{N,+}^d \cap -M_{N,+}^d = \{\vzero\}$.

To study the properties of the periodogram, we first define a linear stationary field.
\begin{definition}
A stationary field $\Yf: \Integer^d \rightarrow \Real$ is linear if
\begin{equation}
\Yf[\vi] = \sum_{\vj\in\Integer^d} W[\vi]\psi[\vi-\vj]\mbox{,}
\end{equation}
where $W: \Integer^d \rightarrow \Real$ is a stationary, zero-mean, white noise field such that $\expect |W[\vi]|^2 = 1, \expect |W[\vi]|^4 < \infty$ and $\psi: \Integer^d \rightarrow \Real$ is a kernel satisfying $\sum_{\vi\in\Integer^d} |\psi[\vi]|\|\vi\|_1^{1/2} < \infty$.
\end{definition}

For such a field, the periodogram is an asymptotically unbiased estimator for the power spectrum:
\begin{lemma}
\label{lemma:phat-expect}
Let $\Yf: \Integer^d \rightarrow \Real$ be a linear field and let $\YfN$ be its restriction $M_N^d$.
Then we have, for $\vk \in M_{N,+}^d$,
\begin{equation}
\label{eq:phat-expect}
\expect \Phat_{\YfN}[\vk] = P_{\Yf}(\vk/N) + O(N^{-1/2})\mbox{.}
\end{equation}
\end{lemma}
The proof is obtained as a generalization of the case $d = 1$ found in Proposition 10.3.1 of \cite{brockwell-davis}.
Recasting the periodogram as the Fourier transform of sums of pairs $\YfN[\vi]\YfN[\vi+\vj]$, the bias is due to not weighting these sums by the number of available pairs.
Reweighting eliminates the bias, but at a significant increase in variance, increasing the mean squared error of the estimator.
For this reason, the biased periodogram estimator is often preferred over its unbiased variant \cite{percival-walden}.

Despite its low bias for large $N$, the periodogram is not suitable for power spectrum estimation due its high variance:
\begin{lemma}
\label{lemma:phat-covar}
Let $\Yf$ and $\YfN$  be defined as in Lemma \ref{lemma:phat-expect} and
\begin{equation}
\delta_N[\vk_1,\vk_2] =
\left\{
\begin{array}{ll}
2 & \vk_1 = \vk_2 \mbox{~and~} \vk_1 \in \frac N 2 \Integer^d \\
1 & \vk_1 = \vk_2 \mbox{~and~} \vk_1 \notin \frac N 2 \Integer^d \\
0 & \vk_1 \neq \vk_2
\end{array}
\right.
\mbox{.}
\end{equation}
We then have, for all $\vk_1, \vk_2 \in M_{N,+}^d$,
\begin{align}
\nonumber
&\cov\left[ \Phat_{\YfN}[\vk_1], \Phat_{\YfN}[\vk_2] \right]\\
\label{eq:phat-covar}
&\,\,\,\,\, = 
\delta_N[\vk_1,\vk_2] P_{\Yf}(\vk_1/N)^2 + O(N^{-1/2})
\mbox{.}
\end{align}
\end{lemma}
The proof again follows the one-dimensional case, Proposition 10.3.2 in \cite{brockwell-davis}.
It relies on the fact that as $N \rightarrow \infty$, $\Phat_{\YfN}[\vk]$ converges to a $\chi^2$-distributed random variable.
Asymptotically, the standard deviation of $\Phat_{\YfN}[\vk]$ is therefore approximately proportional to its expectation $P_{\Yf}(\vk/N)$.

\subsection{The Multitaper Method}
\label{sec:multitaper}

The multitaper method for spectral estimation was introduced by D. J. Thomson as an alternative to the periodogram with reduced variance at the cost of increased bias \cite{thomson}.
It relies on multiplying the sample $\YfN$ by a number of data tapers, computing their periodograms, and averaging.

Given a target spectral resolution $\frac{1}{2N} < W < \frac{1}{2}$, the data tapers for $d = 1$ are given by $K = \lfloor 2NW \rfloor$ discrete prolate spheroidal sequences $v_{N,W}^{(l)}[i]$ for $i \in M_N$, where $l \in 1, \ldots, K$.
Abusing notation slightly, we denote their tensor products by
\begin{equation}
v_{N,W}^{(\vl)}[\vi] = \prod_{p=1}^d v_{N,W}^{(l_p)}[i_p] \quad \vi = (i_1, \ldots, i_d) \in M_N^d\mbox{,}
\end{equation}
where $\vl = (l_1, \ldots, l_d) \in \{1, \ldots, K\}^d$ \cite{hanssen1997multidimensional}.
The associated multitaper estimator is then given by
\begin{equation}
\Phat^{\MT}_{\YfN}[\vk] = \frac{1}{K^d} \sum_{\vl \in \{1, \ldots, K\}^d} \Phat_{v_{N,W}^{(\vl)} \cdot \YfN}[\vk] \mbox{,}
\end{equation}
where $(v_{N,W}^{(\vl)} \cdot \YfN)[\vi] = v_{N,W}^{(\vl)}[\vi] \YfN[\vi]$ for all $\vi \in M_N^d$.

The resolution parameter $W$ specifies the bias-variance trade-off of the estimator.
As $W$ is increased, variance is reduced, but the estimate is smoothed, potentially increasing the bias of the estimator.

\section{Power Spectrum Factor Analysis}
\label{sec:factor}

In certain applications, we do not have just one signal whose power spectrum we would like to estimate, but a set of signals.
To estimate their power spectra, we can apply the methods described in the previous section to each signal individually.
If in addition the signals are identically distributed, we can reduce the variance in the estimate by an ensemble average.

On the other hand, if the signals are not identically distributed, an ensemble average will destroy the variability in the power spectra.
A different approach is applicable if the signals can be written as linear combinations of a small set of fixed random fields.
In this case, factor analysis allows us to identify the subspace spanned by the power spectra of these fixed fields, allowing for increased accuracy by projecting multitaper estimates onto this space.

\subsection{Motivation: Cryo-Electron Microscopy}
\label{sec:cryoem}

Single-particle cryo-EM images are typically very noisy, with signal-to-noise ratios typically below $1/10$ \cite{frank,callaway}.
Many reconstruction algorithms assume that this noise is close to white, but this is rarely the case.
As a result, a whitening step is typically performed prior to reconstruction, in which the noise power spectrum is estimated and used to compute a whitening filter that is then applied.
Power spectrum estimation is typically done by extracting image patches containing mostly noise and averaging their periodograms \cite{bhamre2016denoising}.
Other methods estimate the noise power spectrum as part of the reconstruction algorithm \cite{relion}.

\begin{figure}
\begin{center}
\begin{tikzpicture}
\useasboundingbox (0, 0) rectangle (8, 4.1);
\draw[anchor=south] (2, 3.5) node {Image \#9078};
\node[anchor=south,inner sep=0] at (2, 0.6) {\includegraphics[width=2.8cm]{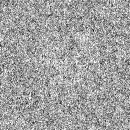}};
\draw[anchor=south] (2, 0) node {\footnotesize (a)};

\draw[anchor=south] (6, 3.5) node {Image \#9935};
\node[anchor=south,inner sep=0] at (6, 0.6) {\includegraphics[width=2.8cm]{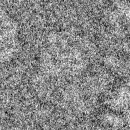}};
\draw[anchor=south] (6, 0) node {\footnotesize (b)};
\end{tikzpicture}
\end{center}
\caption{
\label{fig:cryoem-samples}
Two sample cryo-electron microscopy images with noise power spectra (a) close to white noise and (b) mostly concentrated in the lower frequencies.
Both images are taken from the same experimental dataset \cite{frank70s}.
}
\end{figure}
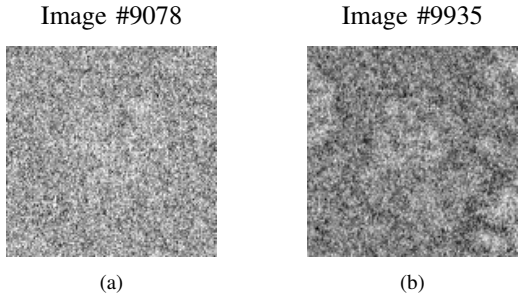

However, the noise distribution typically varies between projection images.
Indeed, different projection images are obtained under different experimental conditions (ice thickness, electron dose, microscope defocus, electron scattering properties of the molecules, and so on) which affect the characteristics of the noise \cite{baxter2009determination,penczek2010chapter,vulovic2013image}.
The variation in the noise signal can be modeled as a random linear combination of fixed noise sources due to background electrons, electrons scattered by the ice, and those inelastically scattered by the molecule.
Experimental images with different noise power spectra are shown in Figure \ref{fig:cryoem-samples}.
As we will see in the following, the power spectra of these noise signals can be accurately estimated using the proposed factor analysis.

\subsection{Problem Setup}
\label{sec:setup}

Let us denote the zero-mean stationary random field of interest by $\Xf: \Integer^d \rightarrow \Real$.
We model it in \eqref{eq:model-def} as a linear combination of $r$ fixed independent linear fields $\Zf_1, \ldots, \Zf_r$, with independent random coefficients $a_1, \ldots, a_r$.

Instead of a single field $\Xf$, we are concerned with a set of $n$ independent copies $\Xf_1, \ldots, \Xf_n$ of $\Xf$ defined by
\begin{equation}
\label{eq:model-s-def}
\Xf_s[\vi] \coloneqq \sum_{l=1}^r a_{s,l} \Zf_{s,l}[\vi]\mbox{,}
\end{equation}
for $s = 1, \ldots, n$.
Here $\Zf_{1,l}, \ldots, \Zf_{n,l}$ and $a_{1,l}, \ldots, a_{n,l}$ are independent and identically distributed copies of $\Zf_l$ and $a_l$, respectively, for $l = 1, \ldots, r$.
The given data consists of restrictions $\XfN_1, \ldots, \XfN_n$ of $\Xf_1, \ldots, \Xf_n$ to $M_N^d$.

The coefficient vectors $\va_1, \ldots, \va_n$ are unknown, but fixed for each signal.
As a result, the power spectrum of interest for $\Xf_s$ is the conditional power spectrum $P_{\Xf_s|\va_s}$, where we have replaced the expectation $\expect$ in \eqref{eq:autocov-def} by the conditional expectation $\expect_{\Zf_{s,1}, \ldots, \Zf_{s,r}|\va_s}$.
Our problem is therefore estimating $P_{\Xf_1|\va_1}, \ldots, P_{\Xf_n|\va_n}$ given $\XfN_1, \ldots, \XfN_n$.

One approach is to use the multitaper method described in Section \ref{sec:multitaper} applied to each signal $\XfN_1, \ldots, \XfN_n$.
However, we can obtain improved estimates by exploiting the fact that, for $s = 1, \ldots, n$,
\begin{equation}
\label{eq:model-s-powspec}
P_{\Xf_s|\va_s}(\vxi) = \sum_{l=1}^r a_{s,l}^2 P_{\Zf_l}(\vxi)\mbox{.}
\end{equation}
That is, the conditional power spectra are contained in a subspace of dimension at most $r$ spanned by $P_{\Zf_1}, \ldots, P_{\Zf_r}$.
These power spectra can thus be described using a small set (at most $r$) of common factors.
We can exploit this to increase the accuracy of a power spectrum estimate by projecting it onto this low-dimensional subspace.
In the following, we propose a method for estimating this subspace.

\subsection{Power Spectrum Covariance Estimation}
\label{sec:covar}

The non-trivial eigenvectors of the covariance matrix $\cov\left[ P_{\Xf_s|\va_s} \right]$ span the subspace containing the conditional power spectra.
To simplify our expressions, we define $P_{\Xf|\va}$ in the same way as $P_{\Xf_s|\va_s}$ for $s = 1, \ldots, n$ and note that these are all identically distributed.
We consider estimating the covariance matrix $\cov\left[ P_{\Xf|\va} \right]$ using the following theorem:
\begin{theorem}
\label{theorem}
Let $\Xf$ and $\Xf_1, \ldots, \Xf_n$ be as in \eqref{eq:model-def} and \eqref{eq:model-s-def}, respectively.
Furthermore, let $\XfN_1, \ldots, \XfN_n$ be the restrictions of $\Xf_1, \ldots, \Xf_n$ to $M_N^d$.
In this case, the quantities
\begin{equation}
\mu_n[\vk] = \frac{1}{n} \sum_{s=1}^n \Phat_{\XfN_s}[\vk]
\end{equation}
and
\begin{equation}
C_n[\vk_1,\vk_2] = \frac{1}{n} \sum_{s=1}^n \Phat_{\XfN_s}[\vk_1] \Phat_{\XfN_s}[\vk_2]\mbox{,}
\end{equation}
satisfy
\begin{equation}
\label{eq:mu_n-approx}
\mu_n[\vk] = P_{\Xf}(\vk/N) + \epsilon_1[\vk]
\end{equation}
and
\begin{align}
\nonumber
&\frac{1}{1+\delta[\vk_1,\vk_2]} C_n[\vk_1,\vk_2] - \mu_n[\vk_1] \mu_n[\vk_2] \\
\label{eq:C_n-approx}
&\, = \cov \left[ P_{\Xf|\va}(\vk_1/N), P_{\Xf|\va}(\vk_2/N) \right] + \epsilon_2[\vk_1,\vk_2]\mbox{,}
\end{align}
where $\vk, \vk_1, \vk_2 \in M_{N,+}^d$.
The expected magnitudes $\expect |\epsilon_1[\vk]|$ and $\expect |\epsilon_2[\vk_1,\vk_2]|$ are each bounded by $O(n^{-1/2} + N^{-1/2})$.
\end{theorem}

\begin{proof}
First, we condition Lemma \ref{lemma:phat-expect} on $\va$ to obtain
\begin{equation}
\label{eq:model-phat-condexpect}
\expect_{\Zf_{1},\ldots,\Zf_{r}|\va} \Phat_{\XfN}[\vk] = P_{\Xf|\va}(\vk/N) + O(N^{-1/2})\mbox{,}
\end{equation}
for $\vk \in M_{N,+}^d$.
Taking expectation with respect to $\va$ then gives
\begin{equation}
\label{eq:model-phat-expect}
\expect \Phat_{\XfN}[\vk] = P_{\Xf}(\vk/N) + O(N^{-1/2})\mbox{.}
\end{equation}
Averaging $\Phat_{\XfN_1}[\vk], \ldots, \Phat_{\XfN_n}[\vk]$ then gives us,
\begin{equation}
\label{eq:model-s-ave}
\mu_n[\vk] = \frac{1}{n} \sum_{s=1}^n \Phat_{\XfN_s}[\vk] = \expect \Phat_{\XfN}[\vk] + \epsilon_1[\vk]\mbox{,}
\end{equation}
according to the central limit theorem, where $\expect |\epsilon_1[\vk]| = O(n^{-1/2})$.
Plugging \eqref{eq:model-phat-expect} into \eqref{eq:model-s-ave} then proves \eqref{eq:mu_n-approx}.

Conditioning Lemma \ref{lemma:phat-covar} with respect to $\va$ gives
\begin{align}
\nonumber
&\cov_{\Zf_{1},\ldots,\Zf_{r}|\va} \left[ \Phat_{\XfN}[\vk_1], \Phat_{\XfN}[\vk_2] \right] \\
\label{eq:model-phat-sub-outer-condexpect}
&\, = \delta[\vk_1,\vk_2] P_{\Xf|\va}(\vk_1/N)^2 + O(N^{-1/2})\mbox{.}
\end{align}
Consequently
\begin{align}
\nonumber
&\expect_{\Zf_{1},\ldots,\Zf_{r}|\va} \left[ \Phat_{\XfN_s}[\vk_1] \Phat_{\XfN_s}[\vk_2] \right] \\
\nonumber
&\, = (1 + \delta[\vk_1,\vk_2]) P_{\Xf|\va}(\vk_1/N) P_{\Xf|\va}(\vk_2/N) + O(N^{-1/2})\mbox{.}
\end{align}
Taking the expectation with respect to $\va$ now gives
\begin{align}
\nonumber
&\expect \left[ \Phat_{\XfN_s}[\vk_1] \Phat_{\XfN_s}[\vk_2] \right] \\
\nonumber
&\, = (1+\delta[\vk_1,\vk_2]) \expect \left[ P_{\Xf|\va}(\vk_1/N) P_{\Xf|\va}(\vk_2/N)\right] + O(N^{-1/2})\mbox{.}
\end{align}
We now have
\begin{align}
\nonumber
&\frac{1}{1 + \delta[\vk_1,\vk_2]} \expect \left[ \Phat_{\XfN}[\vk_1] \Phat_{\XfN}[\vk_2] \right] - P_{\Xf}(\vk_1/N) P_{\Xf}(\vk_2/N) \\
&\, = \cov \left[ P_{\Xf|\va}(\vk_1/N), P_{\Xf|\va}(\vk_2/N)\right] + O(N^{-1/2})\mbox{.}
\end{align}
Replacing the left-hand side expectations with sums using the central limit theorem yields \eqref{eq:C_n-approx}.
\end{proof}

The theorem allows us to estimate the covariance matrix of the conditional power spectra.
Traditionally, this would be done by forming $C_n[\vk_1,\vk_2] - \mu_n[\vk_1] \mu_n[\vk_2]$.
However, due to the $\chi^2$ distribution of the periodogram coordinates, the variances of $C_n$ are larger than they should be, so a correction is needed.
That correction takes the form of $(1 + \delta[\vk_1,\vk_2])^{-1}$.

Let us denote this covariance matrix estimate by
\begin{equation}
\Sigma_n[\vk_1,\vk_2] = \frac{1}{1 + \delta[\vk_1,\vk_2]} C_n[\vk_1,\vk_2] - \mu_n[\vk_1] \mu_n[\vk_2]\mbox{,}
\end{equation}
for $\vk_1, \vk_2 \in M_{N,+}^d$.
Calculating its top eigenvectors allows us to estimate the low-dimensional subspace containing the conditional power spectrum $P_{\Xf|\va}$.
Specifically, these eigenvectors define a linear space that approximates the subspace spanned by $P_{\Zf_1}, \ldots, P_{\Zf_r}$.

\subsection{Linear Projection Estimators}
\label{sec:proj}

We now use the information gained in the previous section to obtain a better estimate of the conditional power spectrum $P_{\Xf|\va}$ from the multitaper estimate $\Phat_{\XfN}^{\MT}$.
It satisfies
\begin{equation}
\Phat_{\XfN}^{\MT} = P_{\Xf|\va} + R^{\MT}\mbox{,}
\end{equation}
where the size of the residual $R^{\MT}$ depends on $N$ and the regularity of $P_{\Xf|\va}$ with respect to the target multitaper resolution $W$ \cite{percival-walden}.
We know that $P_{\Xf|\va}$ is in the subspace spanned by the non-trivial eigenvectors of $\cov\left[P_{\Xf|\va}\right]$.
We can use this to reduce the variance of the residual term.

While the exact covariance is not known to us, we have an estimator $\Sigma_n$.
Let $v_1, \ldots, v_r$ be the leading eigenvectors of $\Sigma_n$.
We note that $r$ does not need to be known beforehand, but can be estimated from the eigenspectrum of $\Sigma_n$.
Here, we set $r$ to be the number of dominant eigenvalues of $\Sigma_n$ by locating the ``knee'' of the spectrum, where there is a significant drop in amplitude from the $r$th eigenvalue to the $(r+1)$th.
In addition, since we expect $\mu_n$ to be contained in the span of $v_1, \ldots, v_r$, we can also check that the projection of $\mu_n$ to that span preserves most of its energy. 

Projecting $\Phat_{\XfN}^{\MT}$ onto the span of $v_1, \ldots, v_r$ yields
\begin{equation}
\Phat^{\MTlinear}_{\XfN}[\vk] = \sum_{l=1}^r v_l[\vk] \langle v_l, \Phat_{\XfN}^{\MT} \rangle \mbox{.}
\end{equation}
This significantly reduces the estimation error $R$ since it will typically be almost orthogonal to $v_l$ and thus $\langle v_l, R \rangle \approx 0$ for $l = 1, \ldots, r$.
The orthogonality follows from the fact that $R$ is approximately isotropically distributed (its coordinates are close to independent random variables), so its inner product with a small set of fixed vectors is small with high probability.
Note that $\Phat^{\MTlinear}$ is not guaranteed to be non-negative.
To remedy this, negative components can be set to zero.
However, this does not greatly affect the error of the estimate.

\section{Numerical Results}
\label{sec:results}

To evaluate our method, we test it on noisy images simulated using \eqref{eq:model-def} for $r = 2$.
The first noise source $\Zf_1$ has its energy concentrated in the low frequencies, with a power spectrum of $P_{\Zf_1}(\vxi) = 2\rect(4\|\vxi\|)$, where $\rect$ is the rectangular function with support $[-1/2, 1/2]$.
The second source $\Zf_2$ has a power spectrum of $P_{\Zf_2}(\vxi) = 1/(1+4\|\vxi\|)$.
These are combined with normally distributed coefficients $a_1, a_2 \sim N(0, 1)$ in \eqref{eq:model-def} to yield $\Xf$.
Figure \ref{fig:sim-results}(a) shows two sample images of size $N = 32$.

\begin{figure*}
\begin{center}
\begin{tikzpicture}
\useasboundingbox (0, 0) rectangle (14, 4);

\node[anchor=south,inner sep=0] at (1, 2.5) {\includegraphics[width=1cm]{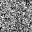}};
\node[anchor=south,inner sep=0] at (1, 1.25) {\includegraphics[width=1cm]{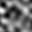}};
\node[anchor=south] at (1, 0.1) {(a)};

\node[anchor=south,inner sep=0] at (4.3, 1) {\includegraphics[width=4cm]{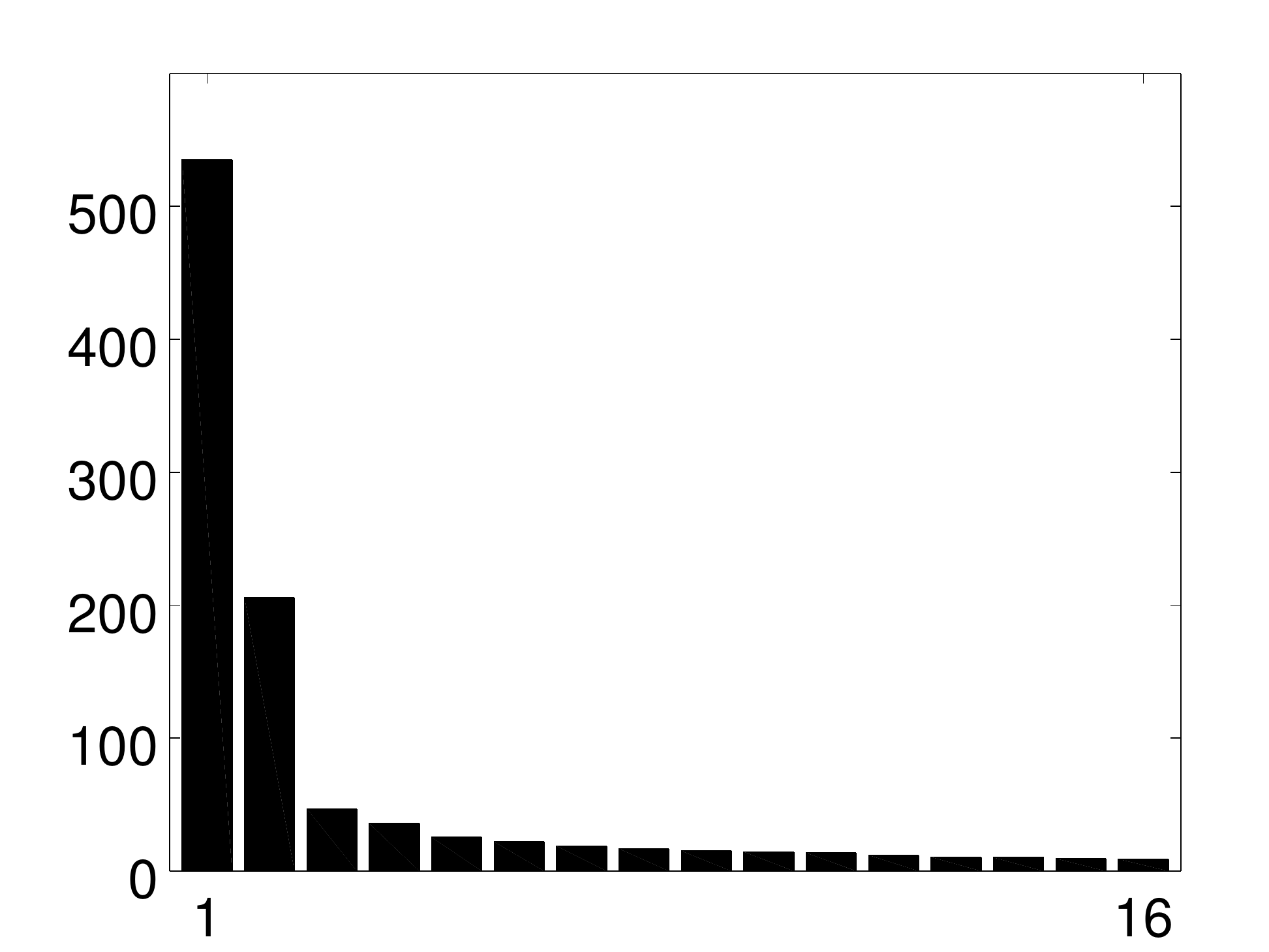}};
\node[anchor=south] at (4.3, 0.7) {\footnotesize $l$};
\node[anchor=south,rotate=90] at (2.5, 2.6) {\footnotesize $\lambda_l$};
\node[anchor=south] at (4.3, 0.1) {(b)};

\node[anchor=south,inner sep=0] at (8.3, 1) {\includegraphics[width=4cm]{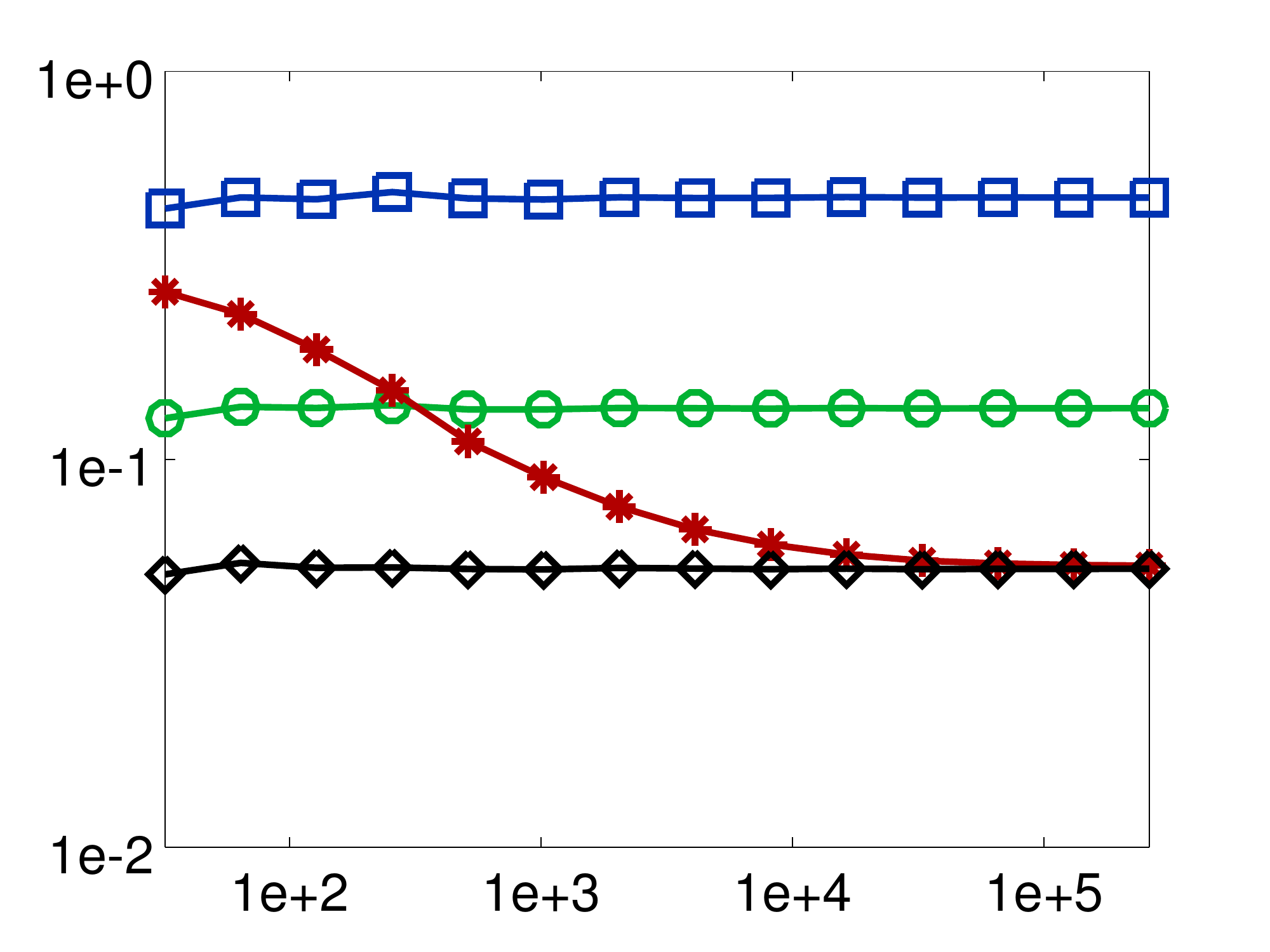}};
\node[anchor=south] at (8.3, 0.7) {\footnotesize $n$};
\node[anchor=south,rotate=90] at (6.5, 2.6) {\footnotesize MAE};
\node[anchor=south] at (8.3, 0.1) {(c)};

\node[anchor=south,inner sep=0] at (12.3, 1) {\includegraphics[width=4cm]{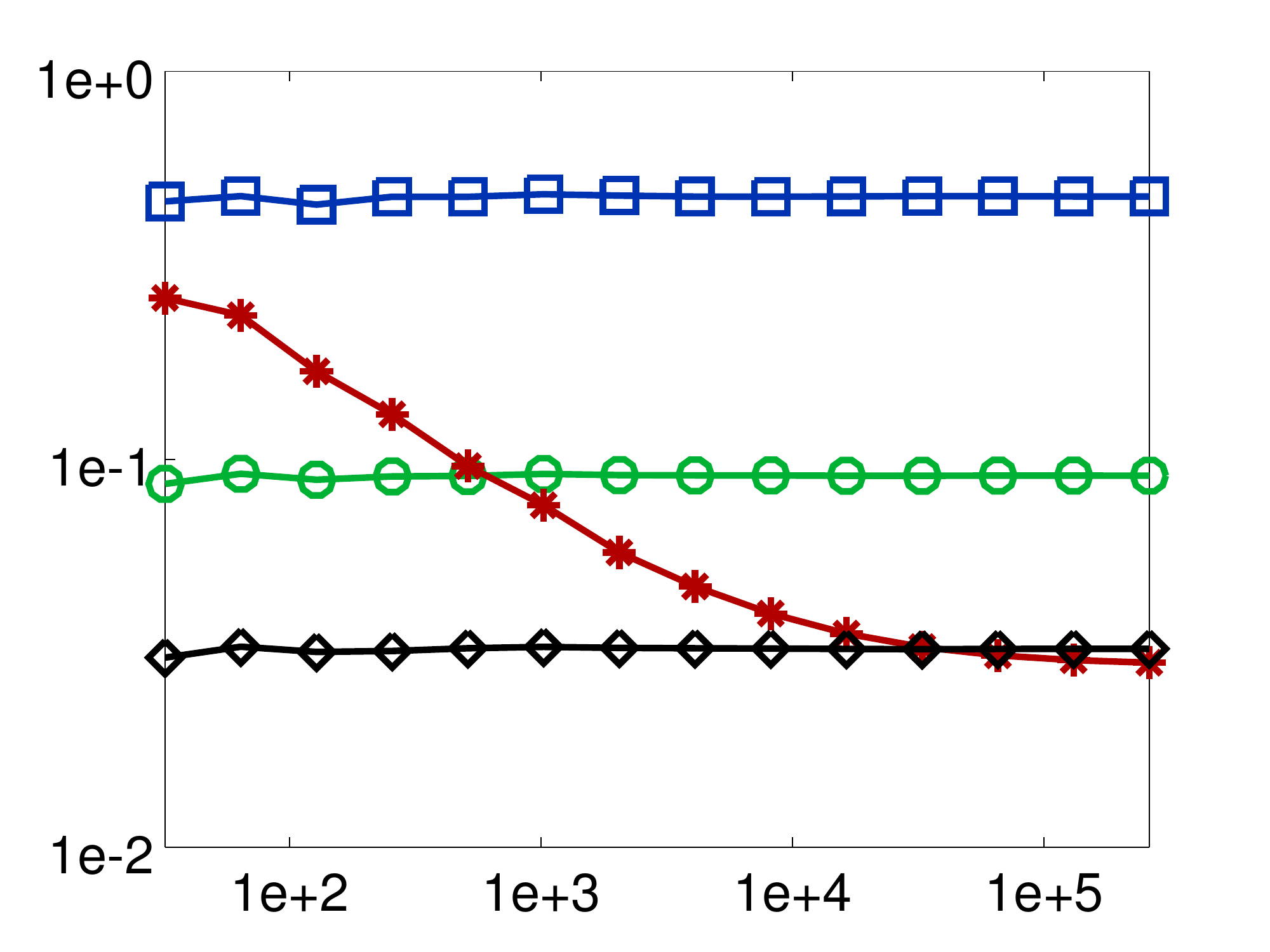}};
\node[anchor=south] at (12.3, 0.7) {\footnotesize $n$};
\node[anchor=south,rotate=90] at (10.5, 2.6) {\footnotesize MAE};
\node[anchor=south] at (12.3, 0.1) {(d)};
\end{tikzpicture}
\end{center}
\caption{
\label{fig:sim-results}
(a) Two sample noise images for $N = 32$ and $r = 2$.
(b) The top $16$ eigenvalues of $\Sigma_n$ for $n = 1024$ images of size $N = 32$.
The estimation error of the conditional power spectra as a function of $n$ for images of size (c) $N = 32$ and (d) $N = 128$.
We have the ensemble average multitaper estimator (blue square), the unprojected multitaper (green circle), the projected multitaper (red star), and the oracle projected multitaper (black diamond).
}
\end{figure*}
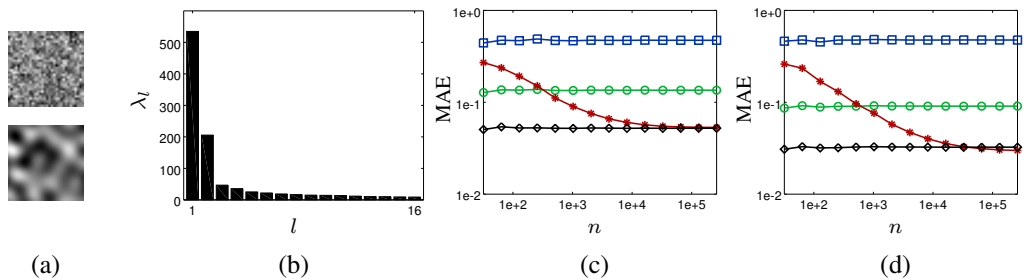

We calculate noise images of different sizes $N$ to study the effect of this parameter on the accuracy of the estimation.
Figure \ref{fig:sim-results}(b) shows the top eigenvalues of $\Sigma_n$ for a dataset of size $n = 1024$ with image size $N = 32$.
There is a good separation of the top two eigenvalues from the bulk of the spectrum, with a spectral gap of $\lambda_2/\lambda_3$ of about $3.4$.

Figure \ref{fig:sim-results}(c,d) shows the mean absolute error (MAE) of $\Phat^\MT_{\XfN_1}, \ldots, \Phat^\MT_{\XfN_n}$ and $\Phat^\MTlinear_{\XfN_1}, \ldots, \Phat^\MTlinear_{\XfN_n}$ with respect to the conditional power spectra $P_{\Xf_1|\va_1}, \ldots, P_{\Xf_n|\va_n}$ as functions of $n$ for image sizes $N = 32$ and $N = 64$, respectively.
The unprojected multitaper estimator was computed with $W = 1/16$ while the others used $W = 1/64$.
For comparison, we plot the baseline estimator of assigning $\mu_n$ to each power spectrum and the oracle projection of the multitaper estimate, where the top $r$ eigenvectors of the population covariance $\cov[P_{\Xf|\va}]$ are used in the definition of $\Phat^\MTlinear$.
For small values of $n$, the covariance estimation fails, so the projected estimator $\Phat^\MTlinear$ performs worse compared to the unprojected variant $\Phat^\MT$.
As $n$ increases, however, the linear projection improves results, eventually converging to the error obtained using the subspace obtained from the population covariance.

\begin{figure}

\begin{center}
\begin{tikzpicture}

\node[anchor=south west,inner sep=0] at (0.3,3.46) {\includegraphics[height=1.90cm]{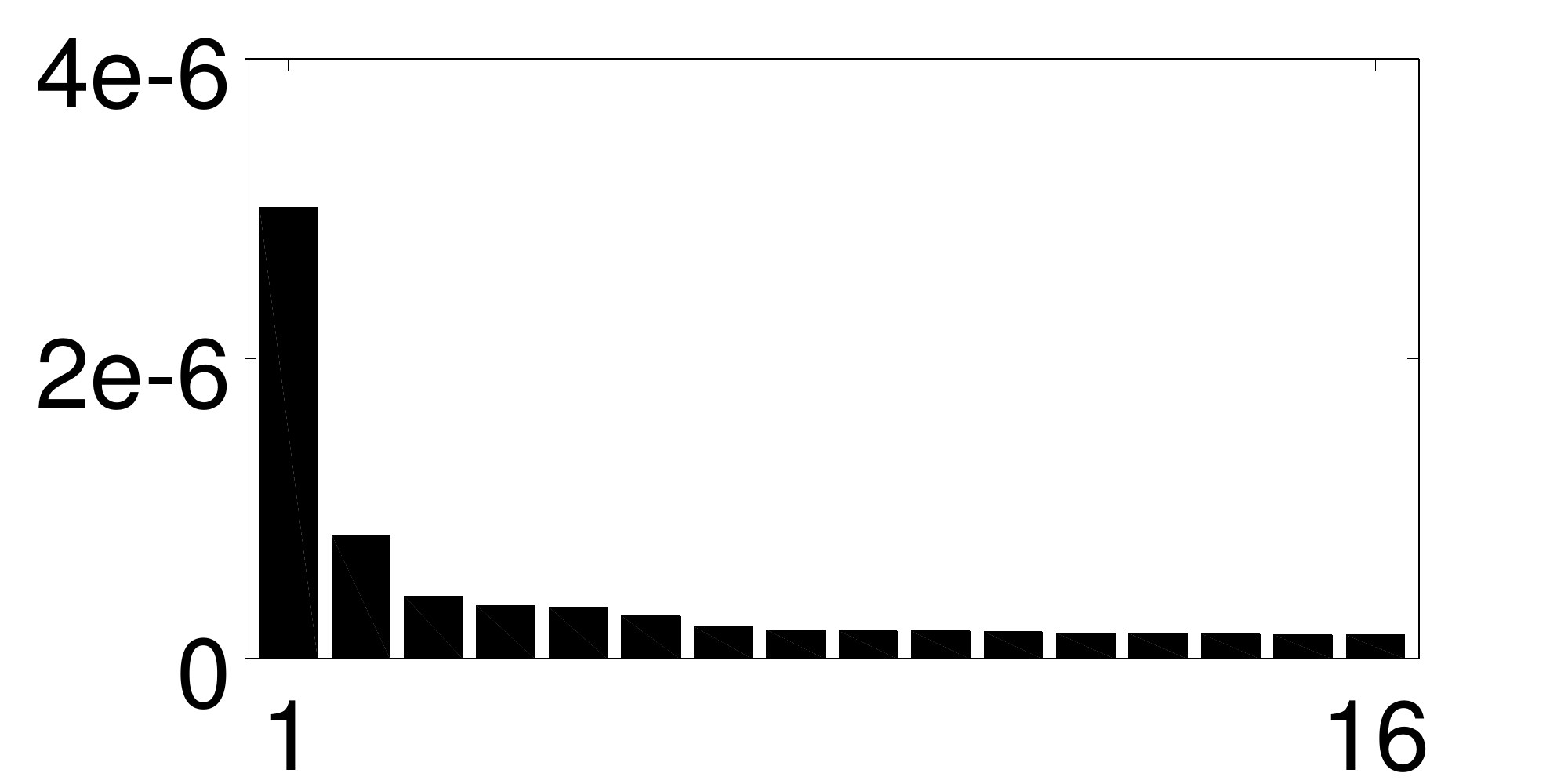}};
\node at (2.2, 3.2) {\footnotesize $l$};
\node[rotate=90] at (0.15, 4.4) {\footnotesize $\lambda_l/N^4$};
\node at (2.2, 2.7) {(a)};

\node[anchor=south west,inner sep=0] at (4.8,3.41) {\includegraphics[height=1.90cm]{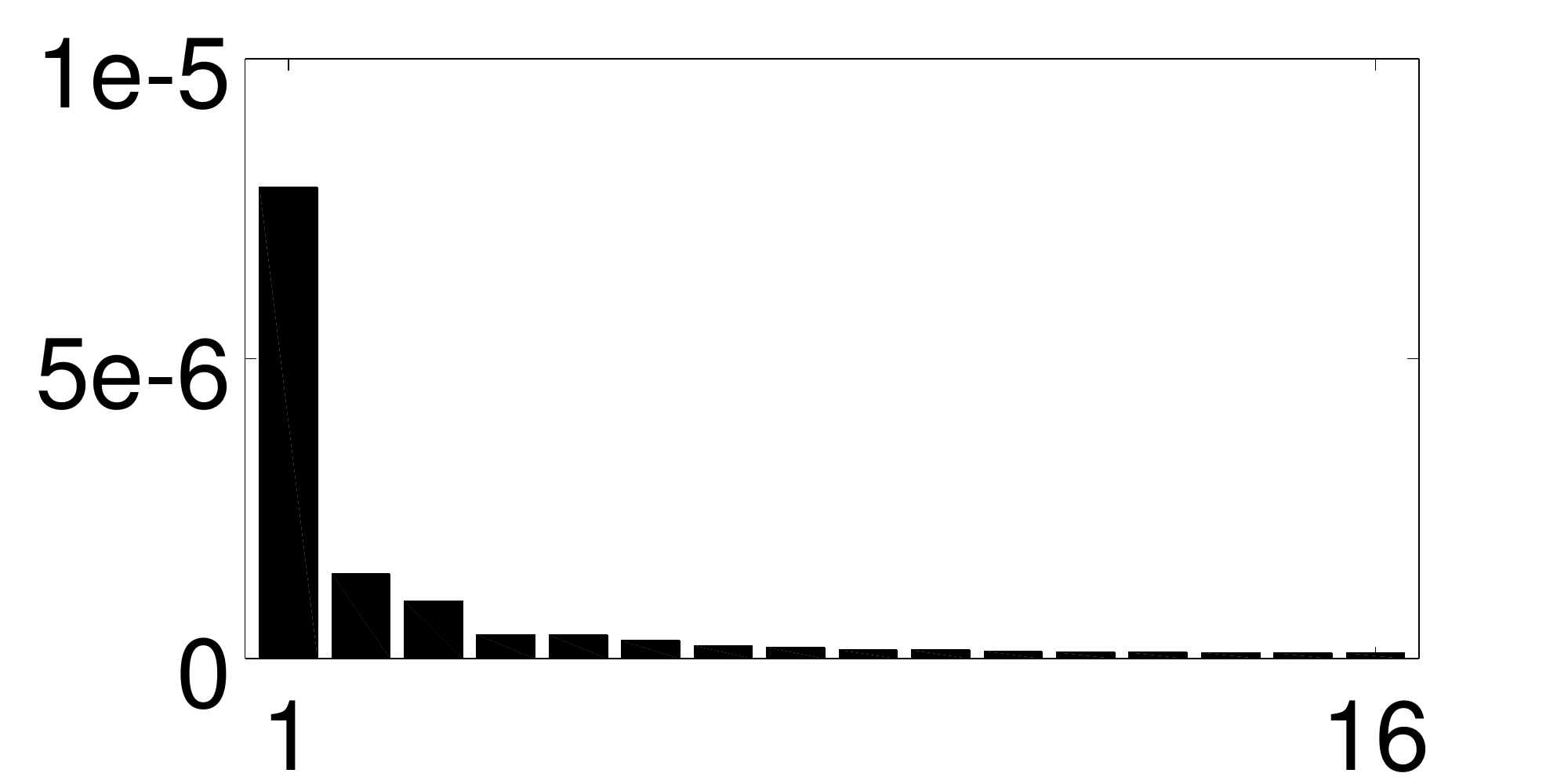}};
\node at (6.7, 3.2) {\footnotesize $l$};
\node[rotate=90] at (4.6, 4.4) {\footnotesize $\lambda_l/N^4$};
\node at (6.7, 2.7) {(b)};

\node[anchor=south west,inner sep=0] at (0.3,0.26) {\includegraphics[height=1.95cm]{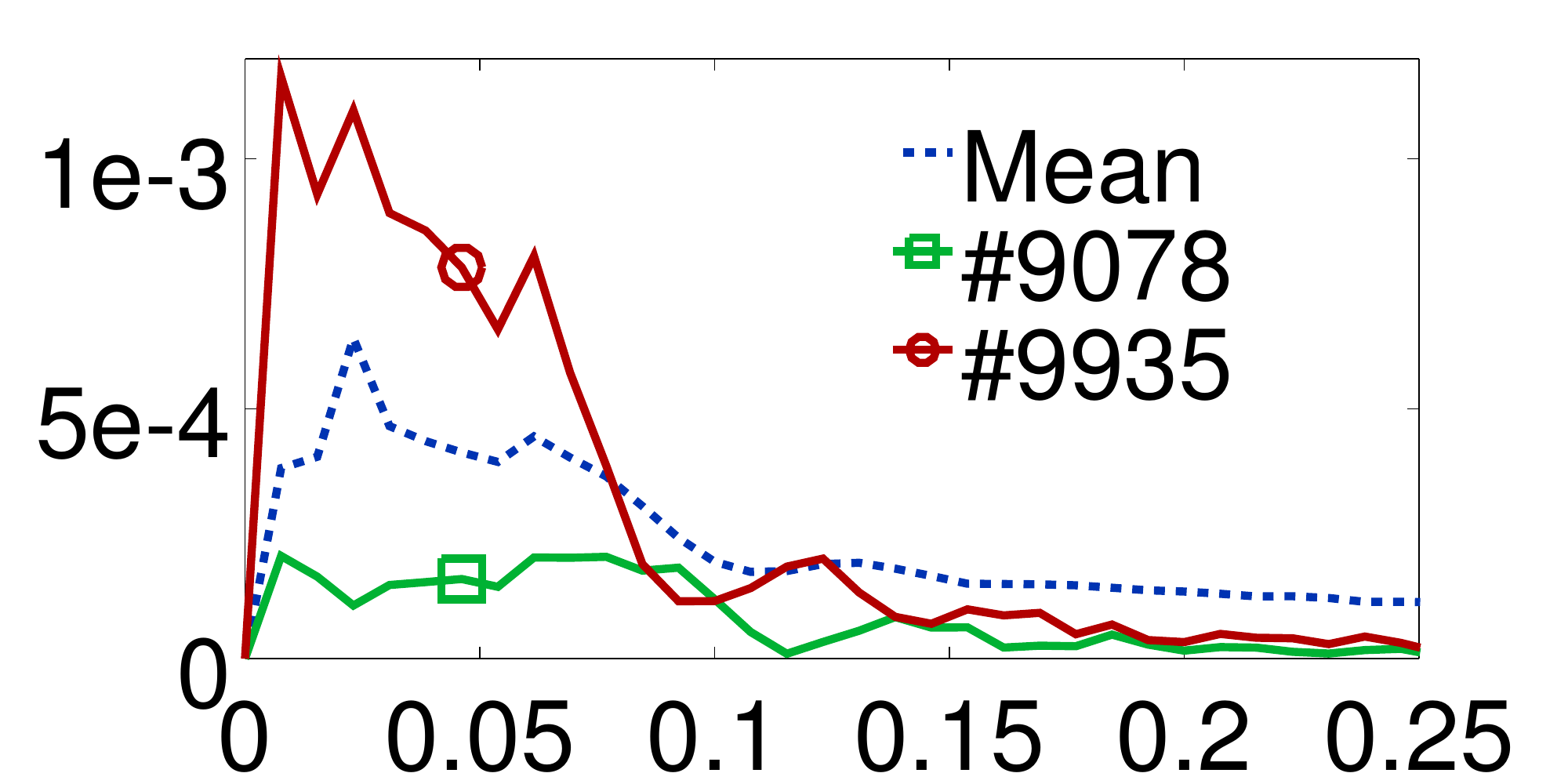}};
\node at (2.2, 0) {\footnotesize $\|\vk\|$};
\node[rotate=90] at (0.15, 1.2) {\footnotesize $\Phat^\MTlinear_{\XfN_s}[\vk]/N^2$};
\node at (2.2, -0.5) {(c)};

\node[anchor=south west,inner sep=0] at (4.8,0.21) {\includegraphics[height=1.95cm]{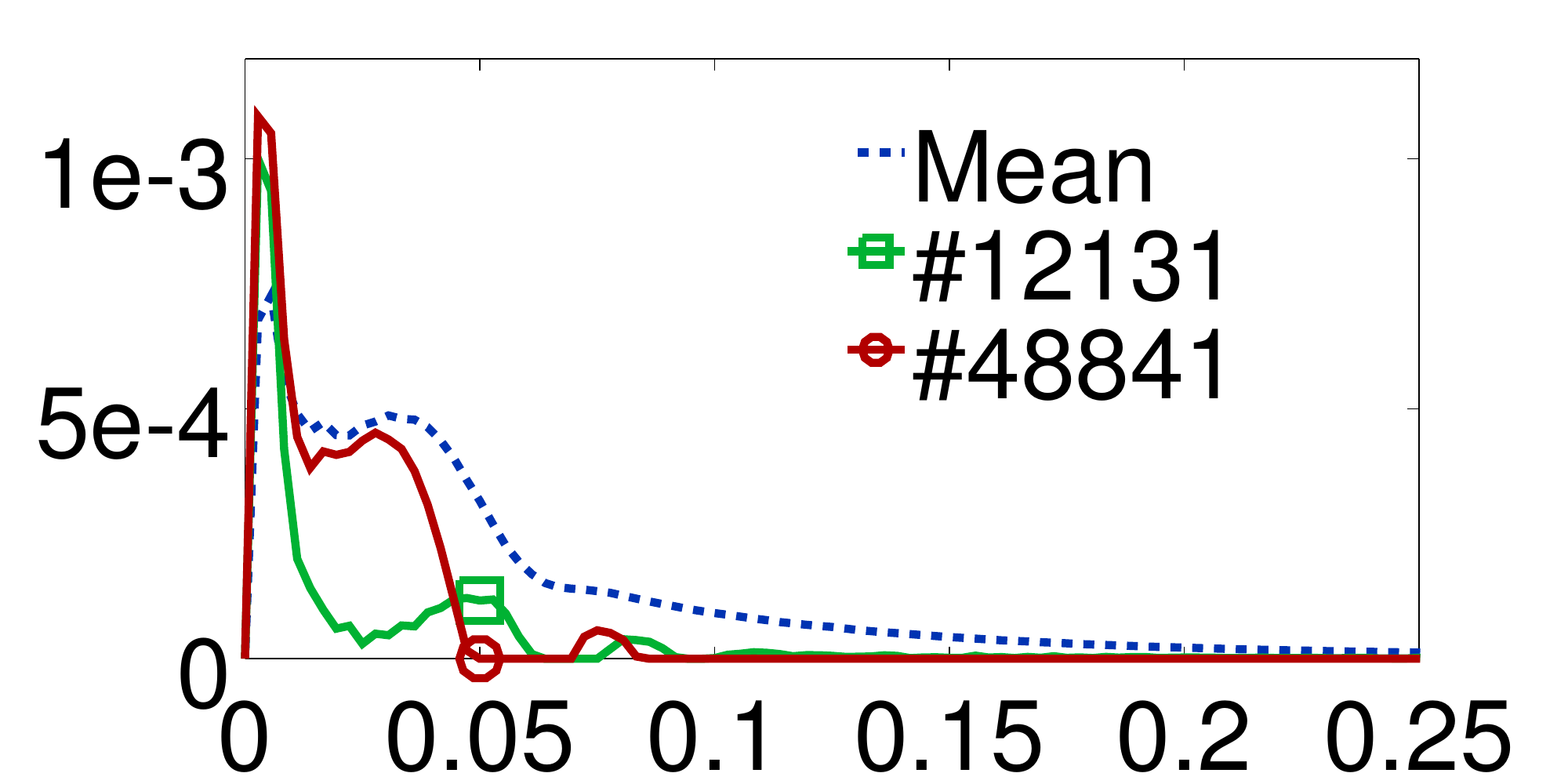}};
\node at (6.7, 0) {\footnotesize $\|\vk\|$};
\node[rotate=90] at (4.6, 1.2) {\footnotesize $\Phat^\MTlinear_{\XfN_s}[\vk]/N^2$};
\node at (6.7, -0.5) {(d)};

\end{tikzpicture}
\end{center}

\caption{
\label{fig:exp-results}
Results experimental datasets.
The top $16$ eigenvalues of $\Sigma_n$ estimated from a dataset depicting a (a) 70S ribosome from E. Coli where $N = 130$ and $n = 10000$ \cite{frank70s} and (b) an 80S ribosome from P. Falciparum where $N = 360$ and $n = 105247$ \cite{scheres80s}.
Sample conditional power spectrum estimates from the (c) 70S ribosome and (d) 80S ribosome datasets.
}
\end{figure}
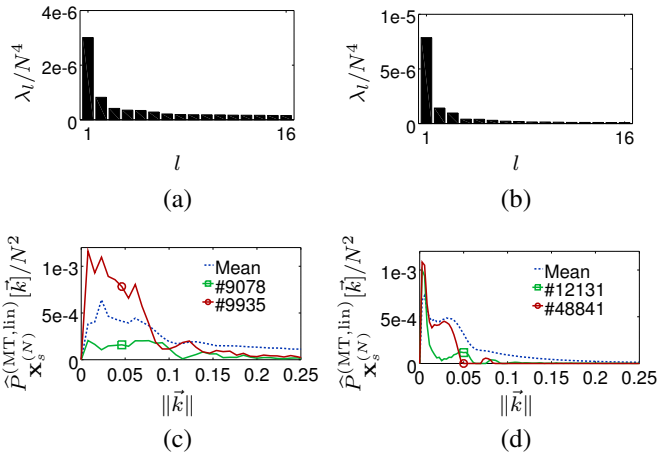

We also evaluate our algorithm using two experimental datasets from cryo-EM, one containing $n = 10000$ images of size $N = 130$ depicting a 70S ribosome \cite{frank70s} and another containing $n = 105247$ images with $N=360$ depicting an 80S ribosome \cite{scheres80s}.
The top eigenvalues of $\Sigma_n$ are shown for the two datasets in Figure \ref{fig:exp-results}(a,b).
As in the simulation, we see that a few eigenvalues dominate, with the estimation making up a separate bulk distribution.
For the first dataset, a value of $r = 2$ seems appropriate, while for the second dataset, we take $r = 3$.
This number of components is in line with most noise models for cryo-EM, which typically include two or three noise components \cite{baxter2009determination,penczek2010chapter,vulovic2013image}.

Two sample conditional power spectra estimates for each dataset are shown in Figure \ref{fig:exp-results}(c,d).
The projection images corresponding to the two power spectra in \ref{fig:exp-results}(c) are the ones shown previously in Figure \ref{fig:cryoem-samples}.
We see that the estimated power spectra are quite reasonable, with the lower-frequency noise image corresponding to the power spectrum concentrated in the low frequencies while the image with ``whiter'' noise has a flatter estimated power spectrum.

\section{Conclusion}
\label{sec:conclusion}

We have introduced a factor analysis model for noise fields.
These arise when the nature of the noise varies between measurements such as in cryo-EM, where the noise is affected by various experimental factors.
To estimate the individual power spectra of each noise image we introduce a new estimation method which relies on approximating the covariance of these spectra.
The method is shown to provide accurate results in both simulated and experimental datasets.

\section*{Acknowledgment}
The authors thank Fred Sigworth for discussions about cryo-EM noise models, which inspired this work. They also thank Frederik Simons for his helpful comments.

\bibliographystyle{IEEEtran}

\bibliography{refs}

\end{document}

%% file: psdsampta-header.tex
\newcommand{\ja}[1]{}
\newcommand{\as}[1]{}

%% file: psdsampta.bbl
\begin{thebibliography}{10}
\providecommand{\url}[1]{#1}
\csname url@samestyle\endcsname
\providecommand{\newblock}{\relax}
\providecommand{\bibinfo}[2]{#2}
\providecommand{\BIBentrySTDinterwordspacing}{\spaceskip=0pt\relax}
\providecommand{\BIBentryALTinterwordstretchfactor}{4}
\providecommand{\BIBentryALTinterwordspacing}{\spaceskip=\fontdimen2\font plus
\BIBentryALTinterwordstretchfactor\fontdimen3\font minus
  \fontdimen4\font\relax}
\providecommand{\BIBforeignlanguage}[2]{{%
\expandafter\ifx\csname l@#1\endcsname\relax
\typeout{** WARNING: IEEEtran.bst: No hyphenation pattern has been}%
\typeout{** loaded for the language `#1'. Using the pattern for}%
\typeout{** the default language instead.}%
\else
\language=\csname l@#1\endcsname
\fi
#2}}
\providecommand{\BIBdecl}{\relax}
\BIBdecl

\bibitem{stoica1997introduction}
P.~Stoica and R.~L. Moses, \emph{Introduction to Spectral Analysis}.\hskip 1em
  plus 0.5em minus 0.4em\relax Prentice Hall, 1997, vol.~1.

\bibitem{bond1997pervasive}
G.~Bond \emph{et~al.}, ``A pervasive millennial-scale cycle in {N}orth
  {A}tlantic {H}olocene and glacial climates,'' \emph{Science}, 278(5341):
  1257--1266 (1997).

\bibitem{harig2012mapping}
C.~Harig and F.~J. Simons, ``Mapping {G}reenland's mass loss in space and
  time,'' \emph{Proc. Natl. Acad. Sci.}, 109(49): 19934--19937 (2012).

\bibitem{delorme2004eeglab}
A.~Delorme and S.~Makeig, ``{EEGLAB}: an open source toolbox for analysis of
  single-trial {EEG} dynamics including independent component analysis,''
  \emph{J. Neurosci. Meth.}, 134(1): 9--21 (2004).

\bibitem{brockwell-davis}
P.~J. Brockwell and R.~A. Davis, \emph{Time Series: Theory and Methods},
  2nd~ed.\hskip 1em plus 0.5em minus 0.4em\relax Springer-Verlag New York,
  1991.

\bibitem{percival-walden}
D.~B. Percival and A.~T. Walden, \emph{Spectral Analysis for Physical
  Applications}.\hskip 1em plus 0.5em minus 0.4em\relax Cambridge University
  Press, 1993.

\bibitem{thomson}
D.~J. Thomson, ``Spectrum estimation and harmonic analysis,'' \emph{Proceedings
  of the IEEE}, 70(9): 1055-1096 (1982).

\bibitem{abreu2017mse}
L.~D. Abreu and J.~L. Romero, ``{MSE} estimates for multitaper spectral
  estimates and off-grid compressive sensing,'' \emph{arXiv:1703.08190}, 2017.

\bibitem{frank}
J.~Frank, \emph{Three-dimensional electron microscopy of macromolecular
  assemblies}.\hskip 1em plus 0.5em minus 0.4em\relax Academic Press, 1996.

\bibitem{callaway}
E.~Callaway, ``The revolution will not be crystallized: a new method sweeps
  through structural biology,'' \emph{Nature}, 525(7568): 172--174 (2015).

\bibitem{hanssen1997multidimensional}
A.~Hanssen, ``Multidimensional multitaper spectral estimation,'' \emph{Signal
  Processing}, 58(3): 327--332 (1997).

\bibitem{bhamre2016denoising}
T.~Bhamre, T.~Zhang \emph{et~al.}, ``Denoising and covariance estimation of
  single particle cryo-{EM} images,'' \emph{J. Struct. Biol.}, 195(1): 72--81
  (2016).

\bibitem{relion}
S.~H.~W. Scheres, ``{RELION}: {I}mplementation of a {B}ayesian approach to
  {C}ryo-{EM} structure determination,'' \emph{J. Struct. Biol.}, 180(3):
  519--520 (2012).

\bibitem{frank70s}
H.~Liao and J.~Frank, ``Classification by bootstrapping in single particle
  methods,'' in \emph{Proc. ISBI}.\hskip 1em plus 0.5em minus 0.4em\relax IEEE,
  169--172 (2010).

\bibitem{baxter2009determination}
W.~T. Baxter, R.~A. Grassucci, H.~Gao, and J.~Frank, ``Determination of
  signal-to-noise ratios and spectral snrs in cryo-em low-dose imaging of
  molecules,'' \emph{J. Struct. Biol.}, 166(2): 126--132 (2009).

\bibitem{penczek2010chapter}
P.~A. Penczek, ``Chapter two: {I}mage restoration in cryo-electron
  microscopy,'' \emph{Methods Enzymol.}, 482: 35--72 (2010).

\bibitem{vulovic2013image}
M.~Vulovi{\'c}, R.~B. Ravelli \emph{et~al.}, ``Image formation modeling in
  cryo-electron microscopy,'' \emph{J. Struct. Biol.}, 183(1): 19--32 (2013).

\bibitem{scheres80s}
W.~Wong, X.~Bai \emph{et~al.}, ``Cryo-{EM} structure of the plasmodium
  falciparum 80s ribosome bound to the anti-protozoan drug emetine,''
  \emph{Elife}, 3: e03080 (2014).

\end{thebibliography}
